\documentclass[11.5pt,oneside,reqno]{amsart}
\usepackage[utf8]{inputenc}
\usepackage[english]{babel}
\usepackage[margin=1in]{geometry}
\usepackage{amsmath,amssymb,amsthm,amsfonts}
\usepackage{csvsimple,array,newclude,longtable}
\usepackage{xcolor}
\usepackage{hyperref}
\hypersetup{
    allcolors = black,
    linkbordercolor = blue,
}
\AtBeginDocument{\hypersetup{pdfborder={0 0 1}}}
\usepackage{cleveref}
\crefformat{section}{\S#2#1#3} 
\crefformat{subsection}{\S#2#1#3}
\crefformat{subsubsection}{\S#2#1#3}
\newtheorem{theorem}{Theorem}
\newtheorem*{utheorem}{Theorem}

\newtheorem{lemma}[theorem]{Lemma}
\newtheorem{defn}{Definition}
\newtheorem{prop}{Proposition}
\newtheorem{conj}{Conjecture}
\newcommand{\ZZ}{\mathbb{Z}}
\newcommand{\NN}{\mathbb{N}}
\newcommand{\chart}[1]{\text{char}(#1)}
\newcommand{\gen}[1]{\text{gen}(#1)}

\newcommand{\Norm}[1]{N(#1)}
\newcommand{\HH}{\mathbb{H}}

\newcommand{\pn}{(p,n)}
\renewcommand{\vec}[1]{\mathbf{#1}}
\theoremstyle{remark}
\newtheorem*{rem}{Remark}
\title{On Characteristics of Hyperfields Obtained as Quotients of Finite Fields}
\author{Antonio Frigo \and Hahn Lheem \and Dylan Liu\\\\Proposed by: Matthew Baker\\ Counselor: Oron Propp}

\begin{document}

\begin{abstract}
Hyperstructures are a natural extension of regular algebraic structures in which one of the operations, known as the hyperoperation, is multivalued; a hyperfield is such an extension on a field. M. Krasner (1962) proved that the quotient $\mathbb{F}_p/G$, where $G$ is a subgroup of units in $\mathbb{F}_p$ is a hyperfield. The characteristic of a field may be explicitly determined from the order of the field, but there are no existing generalizations for determining the characteristic of a hyperfield of the form $\mathbb{F}_p/G$. We show that for odd primes $p$, there exists an explicit form for the characteristic of the hyperfield $\mathbb{F}_p/G$ and $|G|=1,2,3,4$. Finally, we prove a general form of the characteristic for hyperfields where $|G|$ is prime. 
\end{abstract}

\maketitle

\section{Introduction}\label{sec:intro}
Hyperfields are a natural generalization of fields. Instead of the addition operation taking two elements and producing exactly one output, the $\boxplus$ operation takes input of two sets and outputs a set in return. In our paper, we aim to study hyperfields constructed from $\ZZ/p\ZZ$ in the method described by Krasner \cite{krasner}. We consider the motivating question:
\begin{center}
\textit{Characterize the submonoids $S\subseteq \NN$ for which $S$ is the characteristic of a hyperfield. }
\end{center}
The characteristic of a hyperfield is analogous to that of a field, except that the characteristic is a submonoid of $\NN$ rather than a natural number. \par 

Our paper is structured as follows. We first provide the necessary definitions and prove basic propositions in \cref{sec:defs}. Next, we prove explicit forms of generating sets of type $\pn-$hyperfields for specific values of $n$, along with some computational data. We then prove a finite criterion for characteristics when $n$ is prime in \cref{sec:finitecriterion}. Computational data for all $\pn-$hyperfields of primes $p$ up to 200 can be found in Appendix \ref{sec:appen1}.\par
We will prove the following theorems in \cref{sec:gensetspn} and \cref{sec:finitecriterion}:
\begin{utheorem}
Let $\HH$ be an $\pn-$hyperfield with $|G|=1$ and $p\ge 2$. Then $\gen{\HH}=\{p\}$. 
\end{utheorem}

\begin{utheorem}
Let $\HH$ be an $\pn-$hyperfield with $|G|=2$ and $p\ge 3$. Then $\gen{\HH}=\{2, p\}$. 
\end{utheorem}

\begin{utheorem}
Let $\HH$ be an $\pn-$hyperfield with $|G|=3$ and $p\ge 7$. Then there always exists positive integral solutions $(a,b)$ to $a^2-ab+b^2=p$. Let $m,n$ be the two possible values of $a+b$. Then the generating set is in the form $\{3,m,n \}$.
\end{utheorem}
\begin{utheorem}
Let $\HH$ be an $\pn-$hyperfield with $|G|=4$ and $p\ge 5$. Let $(a,b)$ be the positive integral solution to $a^2+b^2=p$ and let $n=a+b$. Then the generating set is in the form $\{2,n \}$.
\end{utheorem}
\begin{utheorem}
Let $q$ be a prime. Let $\zeta_q$ be the primitive $q$-th root of unity. The chracteristic of the $(p,q)$-hyperfield is generated by the set $$\{p,q\} \bigcup \left \{\left. \sum_{i=0}^{q-2} a_i \text{ } \right \vert N\left( {\sum_{i=0}^{q-2} a_i\zeta_q^i}\right)\equiv 0\pmod p \right \},$$such that $a_i$ are integers with $0 \leq a_i < p$.
\end{utheorem}
\subsection{Acknowledgements}
The authors would especially like to thank Oron Propp for providing constant guidance and support throughout our research. We would also like to thank Dr. Matt Baker for proposing this project, as well as Dr. Henry Cohn and Dylan Pentland for their helpful insights. Finally, we thank the PROMYS program and the Clay Mathematics Institute for giving us this research opportunity. 

\section{Definitions, Terminology, and Basic Properties}\label{sec:defs}
We now define some necessary concepts.  
\begin{defn}(Hyperoperation)
Given two non-empty subsets $A$,$B \in S$, we define $$A\boxplus B = \bigcup_{a\in A, b\in B}(a\boxplus b).$$
This operation is closed in $S$ and is associative.
\end{defn}
We can now construct analogs for groups, rings, and fields. 

\begin{defn}(Canonical Hypergroup)
A hypergroup G has the hyperoperation $\boxplus$ and a zero element. The hyperoperation must satisfy the following constraints:
\begin{itemize}
\item (Additive Identity) $0\boxplus x = \{ x\}$ for all $x\in G$.
\item (Hyperinverse) Given $x\in G$, there exists unique $y \in G$ such that $0\in x\boxplus y$.
\item (Reversibility) $x\in y\boxplus z$ if and only if $z \in x \boxplus (-y)$
\end{itemize}
\end{defn}
Since we consider only canonical hypergroups, we will refer to them as a hypergroup for conciseness.  \par 

\begin{defn}(Hyperring)
A hyperring $R$ is a hypergroup with an additional operation $\odot$ and element $I$ that satisfy the following:
\begin{itemize}
\item The tuple ($R,\odot$,$I$) is a commutative monoid.
\item The tuple $(R,\boxplus,0)$ is a commutative hypergroup.
\item (Absorption) For all $x\in R$, $x\odot 0 = 0$.
\item (Distributive) $a\odot (x\boxplus y) = (a\odot x)\boxplus (a\odot y)$ for all $a,x,y,\in R$. 
\end{itemize}
\end{defn}

\begin{prop}(Canonical)
The distributive law implies the reversibility axiom.
\end{prop}
\begin{proof}

Without loss of generality, let $x\in y \boxplus z$.
Note that $0 \in -x \boxplus x \subset -x \boxplus (y \boxplus z) = (-x \boxplus y) \boxplus z = -(x \boxplus -y) \boxplus z$, where the last equality follows because $-1 \odot (x \boxplus y)= (-1 \odot x) \boxplus (-1 \odot y)$ from distributivity. Because additive inverses are unique, this implies that $z \in x \boxplus -y$. An analogous argument proves the reverse direction.
\end{proof}

Finally, we define a hyperfield, which is the focus of the paper.

\begin{defn}(Hyperfield)
A hyperfield is a hyperring $F$ with $0\neq 1$ and every non-zero element of $F$ has a multiplicative inverse. 
\end{defn}

We extend a particular property of fields (its characteristic) to hyperfields.

\begin{defn}(Characteristic)
Let $\HH$ be a hyperfield. The characteristic of $\HH$, which we will denote as $\chart{\HH}$, is the submonoid of $\NN$ consisting of all $n\in \NN$ such that $0\in nx$ for all $x\in F$, where $nx$ is defined to be $\underbrace{x\boxplus x\boxplus\cdots\boxplus x}_{n \text{ }x\text{'s}}$.
\end{defn}

\begin{defn}(Generating set)
We call the finite set $G = \{g_1,g_2, \cdots, g_k\}$ of distinct natural numbers the generating set of a submonoid $S \subseteq \NN$ if $S = \{g_1x_1+g_2x_2 + \cdots + g_kx_k\; |\; g_i \in G, x_i \in \NN \cup \{0\}\}$ and $|G|$ is minimal. For hyperfield (respectively, hyperring) $\HH$, we denote the generating set of $\chart{\HH}$ as $\gen{\HH}$.
\end{defn}
\begin{defn}
(Continuous) We call a submonoid continuous at $k$ if and only if every integer greater than or equal to $k$ is contained in the submonoid.
\end{defn}
We will mostly examine the structure of hyperfields generated in the following manner:
\begin{prop}\label{prop:hyperfield_construction} (Krasner)
For any arbitrary ring (resp. a field) $R$, let $G$ be a subgroup of the group $R^\times$ of units in $R$. Then the set of $R/G$ units for the action of $G$ on $R$ by multiplication produces a hyperring (resp. a hyperfield) structure.
\end{prop}
\begin{proof}
Proven by Krasner in \cite{krasner}.
\end{proof}

\begin{rem}
Note that when refering to $G$, we will use angle brackets to denote the set, namely $\langle$ and $\rangle$.
\end{rem}

\begin{defn} (Type $\pn$-hyperfield)
We let a hyperfield of type $(p,n)$ be of the form $(\ZZ/p\ZZ)/G$ for positive prime $p$ and subgroup $G \subseteq (\ZZ/p\ZZ)^{\times}$. A hyperfield is the $(p,n)-$ hyperfield if the subgroup $G$ has size $n$.
\end{defn} 
\begin{rem}
This always yields a valid hyperfield by Proposition \ref{prop:hyperfield_construction}.
\end{rem}
\begin{defn}
For hyperrings $H = R/G$ of type $(p,n)$, the operation $\odot$ satisfies $[a] \odot [b] = [ab]$, where $a,b \in R$ and $[a]$ is the set generated by the multiplicative action of $G$ on $a$, while the operation $\boxplus$ is such that $[a] \boxplus [b]$ is the set of all $[c] \in H$ that are subsets of $\{x+y | x \in [a], y \in [b]\}$.
\end{defn}
We now prove some useful propositions. 
\begin{prop}\label{prop:checkfirst}
To check whether $n \in char(F/G)$, it suffices to check only that $0 \in n[1]$
\end{prop}
\begin{proof}
We see first that $0 \in n[1]$ is a necessary condition for $n \in \chart{F/G}$, as $1 \in F$. Now we show that it is sufficient: Let $a \in F$. Then, $a \odot (n[1])=n[a]$ by distributivity. Since $0$ is the absorption element, $0 \in n[a]$, we're done. Note that this proof extends to hyperrings as well.
\end{proof}
\begin{theorem}\label{thm:partition}
If $G= \langle g_1,g_2,\dots,g_m\rangle$, then $\sum_{i=1}^ma_ig_i=0$ in the ring implies that $\sum_{i=1}^ma_i \in \chart{F/G}$, where the $a_i \in \ZZ_{\ge 0}$. Additionally, $\sum_{i=1}^ma_i \in \chart{F/G}$ only if there exists an $m$-tuple $(a_1,a_2,\dots,a_m)$ such that $\sum_{i=1}^ma_ig_i=0$.
\end{theorem}
\begin{proof}
It suffices to consider the case where $G$ acts upon $1$ by Proposition \ref{prop:checkfirst}.\par
Since we are taking the setwise sum, we see that $0\in n[1]$ if and only if we can add $n$ of the elements in the group together to obtain zero. So, we add some combination of units together, with $a_i$ instances of $g_i$ in the sum, for a total of $n=\sum_{i=1}^ma_i$ summands. Since we assumed that adding the combination of $g_i$ yields zero, $n = \sum_{i=1}^ma_i\in \chart{F/G}.$ 
\end{proof}

\section{Generating sets of $(p,n)$-hyperfields}\label{sec:gensetspn}

\subsection{Computations}
Programs were written in \texttt{Magma}\cite{magma}, \texttt{C} and \texttt{Python} to compute the generating sets of $\pn$-hyperfields. Table \ref{table:geq1} contain the generating set for the $\pn$-hyperfield for some primes $p$ and when $n=1,2,3,4$. Tables for all $(p,n)$-hyperfields for $p<200$ is included in the Appendix, though data for all primes less than 1500 was calculated. \par
\begin{table}[ht]
\caption{$|G|=1,2,3,4$}
    \begin{tabular}{| c c |c c| c c| c c |}
    \hline
    \multicolumn{2}{|c|}{$|G|=1$} &\multicolumn{2}{c}{$|G|=2$} &\multicolumn{2}{|c|}{$|G|=3$} &\multicolumn{2}{|c|}{$|G|=4$}\\ 
    $p$ & Generating Set & $p$ & Generating Set & $p$ & Generating Set & $p$ & Generating Set   \\
    \hline
2	&	$\{2\}$	&	3	&	$\{2, 3\}$	&	7	&	$\{3, 4, 5\}$&5	&	$\{2, 3\}$	\\
3	&	$\{3\}$	&	5	&	$\{2, 5\}$	&	13	&	$\{3, 5, 7\}$&13	&	$\{2, 5\}$	\\
5	&	$\{5\}$	&	7	&	$\{2, 7\}$	&	19	&	$\{3, 7, 8\}$&17	&	$\{2, 5\}$	\\
7	&	$\{7\}$	&	11	&	$\{2, 11\}$	&	31	&	$\{3, 7, 11\}$&29	&	$\{2, 7\}$	\\
11	&	$\{11\}$	&	13	&	$\{2, 13\}$	&	37	&	$\{3, 10, 11\}$&37	&	$\{2, 7\}$	\\
13	&	$\{13\}$	&	17	&	$\{2, 17\}$	&	43	&	$\{3, 8, 13\}$&41	&	$\{2, 9\}$	\\
17	&	$\{17\}$	&	19	&	$\{2, 19\}$	&	61	&	$\{3, 13, 14\}$&53	&	$\{2, 9\}$	\\
19	&	$\{19\}$	&	23	&	$\{2, 23\}$	&	67	&	$\{3, 11, 16\}$&61	&	$\{2, 11\}$	\\
23	&	$\{23\}$	&	29	&	$\{2, 29\}$	&	73	&	$\{3, 10, 17\}$&73	&	$\{2, 11\}$	\\
29	&	$\{29\}$	&	31	&	$\{2, 31\}$	&	79	&	$\{3, 13, 17\}$&89	&	$\{2, 13\}$	\\
31	&	$\{31\}$	&	37	&	$\{2, 37\}$	&	97	&	$\{3, 14, 19\}$&97	&	$\{2, 13\}$	\\
37	&	$\{37\}$	&	41	&	$\{2, 41\}$	&	103	&	$\{3, 13, 20\}$&	101	&	$\{2, 11\}$	\\
41	&	$\{41\}$	&	43	&	$\{2, 43\}$	&	109	&	$\{3, 17, 19\}$&	109	&	$\{2, 13\}$	\\
43	&	$\{43\}$	&	47	&	$\{2, 47\}$	&	127	&	$\{3, 19, 20\}$&    113	&	$\{2, 15\}$	\\
47	&	$\{47\}$	&	53	&	$\{2, 53\}$	&	139	&	$\{3, 16, 23\}$&	137	&	$\{2, 15\}$	\\
53	&	$\{53\}$	&	59	&	$\{2, 59\}$	&	151	&	$\{3, 19, 23\}$&	149	&	$\{2, 17\}$	\\
    \hline
    \end{tabular}
\label{table:geq1}
\end{table}
\subsection{Some useful lemmas}
\begin{lemma}\label{lem:existence_generator}
For any subgroup of units of $\ZZ/p\ZZ$, there is a generator of the subgroup. In particular, if the subgroup is size $|G|$, then a generator of the subgroup is $g^{\frac{p-1}{|G|}}$, where $g$ is a generator of $(\ZZ/p\ZZ)^{\times}$.
\end{lemma}
\begin{proof}
Consider two elements in $G$, namely $g^{da}$ and $g^{db}$ where $g$ is a generator in $\ZZ/p\ZZ$ and $(a,b)=1$. Since there always exists a unique multiplicative inverse in the subgroup, by Bezout's famous lemma, $g^d$ is also in the subgroup, which generates both $g^{da}$ and $g^{db}$. We can continue this process inductively, ultimately showing that there is a generator for all elements in $G$. The second half of the statement follows immediately.
\end{proof}
\begin{lemma}
For any $\ZZ/p\ZZ$ with a subgroup $G=\langle g_0,g_1,\dots,g_m\rangle $ of $(\ZZ/p\ZZ)^{\times}$, if $kp \ge \Sigma b_ig_i$ for non-negative integers $b_i$ and a positive integer $k$, then $s \in S$ iff $s=kp-\Sigma{b_i(g_i-1)}$.
\end{lemma}
\begin{proof}
We first note that $(a_1,a_2,\dots,a_m)=(kp,0,\dots,0)$ for any arbitrary natural number $k$ causes $\Sigma_{i=1}^m a_ig_i=kp=0$, so $kp+0+\dots+0=kp \in S$. \par We can replace any $(a_1,\dots,a_i,\dots,a_m)$ with $(a_1-g_i,\dots,a_i+1,\dots,a_m)$, which still yields a sum of $\Sigma_{i=1}^m a_ig_i=kp=0$, but instead causes $kp-g_{x_1}+1-g_{x_2}+1-\dots-g_{x_n}+1 \in S$ for some sequence of $x_i$, where $1 \le x_i \le m$, as long as $a_1 \ge 0$. This algorithm gives all possible solutions to $\Sigma_{i=1}^m a_ig_i=kp$, so the reverse direction holds. We can simplify this statement to the one stated in the lemma, where $b_j$ represents the number of times some $x_i=j$.
\end{proof}

\begin{lemma}\label{lem:continuity}
For any $\ZZ/p\ZZ$ and a nontrivial subgroup of units $G$, $\chart{(\ZZ/p\ZZ)/G}$ is continuous at $p-1$.
\end{lemma}
\begin{proof}
Because of the nontriviality of the subgroup, there exists an element $g_i\in G$ such that $g_i-1$ is a unit modulo $p$. Then, as $b_i$ ranges from 0 to $p-1$, $s=kp-\Sigma b_i(g_i-1)$ hits every residue modulo $p$, and since $kp \ge \Sigma b_ig_i$, this implies that $s \ge p-1 \ge b_i$, so any value greater than or equal to $p-1$ is valid.
\end{proof}
\begin{lemma}\label{lem:twosolsfor3case}
The equation $a^2-ab+b^2$ has two distinct solutions $(a,b)\in\NN\cup \{0\}$.
\end{lemma}
\begin{proof}
To do this, we first prove that $a^2-ab+b^2=p$ has two distinct solutions up to order in the natural numbers. Consider the Eisenstein integer $a+b\omega \in \ZZ[\omega]$ and its norm, namely $\Norm{a+b\omega}=a^2-ab+b^2=(a+b\omega)(a+b\omega^2)$. Because $\Norm{a+b\omega}$ is a rational prime, $a+b\omega$ and $a+b\omega^2$ must be primes in $\ZZ[\omega]$. \par Since $\ZZ[\omega]$ has unique prime factorization up to associates, we see that $a+b\omega$ and $a+b\omega^2$ and their associates are the only Eisenstein integers whose norms satisfy $a^2-ab+b^2=p$. If we assume without loss of generality $a>b$, then the only two solutions to $a^2-ab+b^2=p$ are $(a,b)$ and $(a,a-b)$.
\end{proof}
\subsection{Explicit forms of generating sets}
We have the following four theorems:
\begin{theorem}
Let $\HH$ be an $\pn-$hyperfield with $|G|=1$ and $p\ge 2$. Then $\gen{\HH}=\{p\}$.    
\end{theorem}

\begin{proof}
Using Theorem \ref{thm:partition}, we want all $n \in \NN$ such that $n\times1=0$ in $\ZZ/p\ZZ$, since $G=\{1\}$. This implies that $a=0$, so $a$ must be any positive multiple of $p$, leading to the generating set $\{p\}$.
\end{proof}

\begin{theorem}
Let $\HH$ be an $\pn-$hyperfield with $|G|=2$ and $p\ge 3$. Then $\gen{\HH}=\{2, p\}$.    
\end{theorem}
\begin{proof}
Since $G=\langle-1,1\rangle,$ it suffices by Theorem \ref{thm:partition} to characterize $a+b$, where $a-b=0$. We can then let $a=b+kp$, so $a+b=2b+pk$. Since $b$ and $k$ can be any natural number, the generating set of the characteristic is $\{2,p\}.$
\end{proof}

\begin{theorem}\label{thm:cardG3}
Let $\HH$ be an $\pn-$hyperfield with $|G|=3$ and $p\ge 7$. Then there always exists positive integral solutions $(a,b)$ to $a^2-ab+b^2=p$. Let $m,n$ be the two possible values of $a+b$. Then the generating set is in the form $\{3,m,n \}$.
\end{theorem}
\begin{proof}
Let $F = \ZZ/p\ZZ$, with prime $p \equiv 1 \pmod{3}$, and $G= \langle \omega, \omega^2, 1 \rangle$, where $\omega$ is a third root of unity. Since $1\times\omega +1\times\omega^2+1\times1=0$, we see that $1+1+1=3$ is an element of $\gen{F/G}$. To find the other generators of the generating set, it suffices to only consider the equation $a\omega+b=0$. \par   
Consider the equation $a_1\omega+a_2\omega^2+a_3=0$. Let $m=\min(a_1,a_2,a_3)$. We know that since $\omega+\omega^2+1=0$, we have $m(\omega + \omega^2 + 1) = 0$. Hence
\begin{align*}
a_1\omega+a_2\omega^2+a_3 - m(\omega + \omega^2 + 1)&=0\\
(a_1-m)\omega + (a_2-m)\omega^2+(a_3-m)&=0. 
\end{align*}
Since $m=\min(a_1,a_2,a_3)$, one of these terms is zero. We can then multiply by some power of $\omega$ to yield an equation in the form $a\omega+b=0$.\par

Then, we have $a\omega \equiv -b$ (mod p), so cubing yields $a^3 \equiv -b^3$ (mod p). So, $p$ necessarily divides $a^3+b^3=(a+b)(a^2-ab+b^2)$ if we want $a+b \in \chart{F/G}$. However, we know if $p|(a+b)$, then $a+b \ge p$, so by Lemma \ref{lem:continuity}, $a+b \in \chart{F/G}$ automatically. Hence we only need to consider $p|(a^2-ab+b^2)$ for generators of $\chart{F/G}$. \par

In fact, we claim that the two values of $a+b$ such that $a^2-ab+b^2=p$, are not only in $\chart{F/G}$ but also the other two elements of the generating set of $\chart{F/G}$.\par

Without loss of generality, let $a> b$. Let $(a,b), (a,a-b)$ are the solutions to $x^2-xy+y^2=p$ (we know these are the only two by Lemma \ref{lem:twosolsfor3case}). It suffices to show that if all pairs of positive integers $(c,d)$ satisfy $c^2-cd+d^2 = kp$ for integer $k>1$, then $\max(a+b,2a-b) \leq c+d$. \par 

We first show that $k \leq 3$. Suppose $k>3$. Noting $c,d$ are positive, we have $$(c+d)^2 > c^2-cd+d^2 = kp \implies c+d > \sqrt{kp}.$$

Furthermore, $$a^2-ab+b^2 = (a-b)^2 + ab = p \implies ab < p.$$ 

Thus, $$a^2-ab+b^2=(a+b)^2-3ab=p \implies (a+b)^2 = p+3ab < 4p \implies a+b < 2\sqrt{p}.$$ 

But for integer $k>3$, we have $$a+b < 2\sqrt{p} \le \sqrt{kp} < c+d.$$ 

Hence if $k>3$, then $\max(a+b,2a-b)\le c+d$.\par 

Furthermore, $c^2-cd+d^2 = kp = k(a^2-ab+b^2)$, so since norms are multiplicative $k$ must be the norm of some Eisenstein integer $f+g\omega$.  We have that $k\neq 2$ since $f^2-fg+g^2 = (f+g)^2 - 3fg \equiv 2 \pmod{3}$ has no solutions. \par

Now consider $k=3$. By process of elimination from above, $c^2-cd+d^2 = 3p$. \par 

Let $x+y\omega \in \ZZ[\omega]$ and $\Norm{x+y\omega} = 3$. From above, $c+d\omega = (x+y\omega)(a+b\omega)$ for positive integers $a,b,c,d$. Again, without loss of generality let $a>b$. (If $a<b$, then take the associate $-\omega(a+b\omega)=b+(b-a)\omega$.) As $\Norm{1+\omega} = 3$, $x+y\omega$ must be some associate of $1+\omega$. Taking all products of $(a+b\omega)$ by an associate of $1+\omega$ yield
$$(c,d) = (a+b,2b-a),(a-2b,2a-b),(2a-b,b+a).$$
If $a>2b$, then the minimum value of $c+d$ such that $c,d >0$ is $(a-2b)+(2a-b) = 3a-3b$. However, $$a>2b \implies 2a>4b \implies a+b < 3a-3b = c+d.$$ If $a<2b$, then $\min(c+d)$ becomes $(a+b)+(2b-a) = 3b$, and $a<2b \implies a+b < 3b = c+d$.\par 
Therefore, if $p|(x^2-xy+y^2)$ for positive integers $x,y$, then $\min(x+y)$ is achieved when $x^2-xy+y^2 = p$. By Lemma \ref{lem:twosolsfor3case}, there are exactly two solutions $(a_1,b_1), (a_2,b_2)$ to this equation, and hence $a_1+b_1$ and $a_2+b_2$ are generators of $\chart{F/G}$.
\end{proof}

\begin{theorem}\label{thm:cardG4}
Let $\HH$ be an $\pn-$hyperfield with $|G|=4$ and $p\ge 5$. Let $(a,b)$ be the positive integral solution to $a^2+b^2=p$ and let $n=a+b$. Then the generating set of $\gen{\chart{\HH}}$ is $\{2,n \}$.
\end{theorem}
\begin{proof}
Let $F = \ZZ/p\ZZ$ for prime $p \equiv 1 \pmod{4}$. By Lemma \ref{lem:existence_generator}, let $g$ be the generator of the subgroup $G$. Then $g$ is a fourth root of unity, so $G = \langle g, g^2, g^3, g^4  \rangle = \langle g, -1, -g, 1 \rangle$. If we wish to find the characteristic of $\HH$, by Proposition \ref{prop:checkfirst}, it suffices to find all positive integers $n$ such that $0 \in n[1] = n\{ g,-1,-g,1 \}$. \par

Because $1 \cdot 1 + 1 \cdot (-1) = 0$, $1+1=2$ must be a generator. Then, there can only be one other possible generator, which can be found by considering only $ag+b=0$. To see why, note that we can take the equation $a_1g-a_2-a_3g+a_4=0$ and subtract $0=\min(a_1,a_3)(g-g)+\min(a_2,a_4)(-1+1)$ (thus $a_1+a_2+a_3+a_4 - 2(\min(a_1,a_3) + \min(a_2,a_4) \in \chart{F/G}$) and choose the more convenient value of $g$. Then, notice that $ag \equiv -b$ mod $p$, so $-a^2 \equiv b^2$ mod $p$. Thus, $p$ necessarily divides $a^2+b^2$. In fact, we claim that the value of $a+b$ where $a^2+b^2=p$ is not only in $\chart{F/G}$ but also the other generator of $\chart{F/G}$. \par

Note that the maximum value of $a+b$ when $a^2+b^2=p$ occurs when $a=b$, so $\max(a+b)=\sqrt{2p}$. The minimum value of $a+b$ when $a^2+b^2=kp$, where $k>1$, however, occurs when one of $a$ or $b$ is equal to zero, so $\min(a+b)=\sqrt{kp} \ge \sqrt{2p}$. So, the other generator is the smallest value of $a+b$ when $p|a^2+b^2$, which occurs when $a^2+b^2=p$. It is well-known that the pair of natural numbers $(a,b)$ is the unique solution up to ordering, so the value of $n$ is also unique.
\end{proof}

\begin{conj}
For a given generating set $\{2,n\}$ where $n$ is a positive odd integer, then there must exist a hyperfield $\HH$ such that $\gen{\HH} = \{2,n \}$.  
\end{conj}
\begin{rem}
By Theorem \ref{thm:cardG4},  we have that $a^2+b^2 = (a+bi)(a-bi)=p$. Note how $a^2+b^2$ is the norm form in the Gaussian integers $\ZZ[i]$. Let $\pi\in \ZZ[i]$. It is a canonical fact that if $\Norm{\pi}$ is prime, then $\pi$ is a prime in the Gaussian integers. Since $a+b=n$, the conjecture reduces to showing that there always exists a Gaussian prime that is $n$ away from the line $y=x$ for all odd $n$. This is a very explicit statement regarding the distribution of the Gaussian primes, leading us to believe that this is a very difficult conjecture to prove. \par

We have verified this result computationally up to $n=89441$.     
\end{rem}

\section{Finite Criterion for Generating Set of  \texorpdfstring{$(p,n)$}-hyperfield}\label{sec:finitecriterion}
Theorems \ref{thm:cardG3} and \ref{thm:cardG4} provide an explicit method for computing generators from solutions to bivariate polynomials. We aim to generalize these theorems and find, for every $n$ such that $|G| = n$, a corresponding polynomial whose solutions can compute generators. We first give a finite criterion for when $n$ is prime, before generalizing to all positive integers $n$. \par

We first define the norm of a cyclotomic field.
\begin{defn}
Let $f(\zeta_n) = f_0 + f_1\zeta_n + \cdots + f_k\zeta_n^k \in \ZZ[\zeta_n]$, where $\zeta_n$ is a primitive $n$-th root of unity. The \textit{norm} of $f(\zeta_n)$ is the integer
        $$\Norm{f(\zeta_n)} = \prod_{i \in (\ZZ/n\ZZ)^{\times}} f(\zeta_n^i).$$
\end{defn}
In other words, the norm of $f \in \ZZ[\zeta_n]$ is the product of all Galois conjugates of $f$, including itself. \par 

We are now prepared to construct a generating set of $(p,q)$-hyperfields for prime $q$.
\begin{theorem}\label{thm:genprime}
Let $q$ be a prime. Let $\zeta_q$ be the primitive $q$-th root of unity. The chracteristic of the $(p,q)$-hyperfield is generated by the set $$\{p,q\} \bigcup \left \{\left. \sum_{i=0}^{q-2} a_i \text{ } \right \vert N\left( {\sum_{i=0}^{q-2} a_i\zeta_q^i}\right)\equiv 0\pmod p \right \},$$such that $a_i$ are integers with $0 \leq a_i < p$.
\end{theorem}

\begin{proof}

Let $S$ be the set defined above and $T$ be the minimal generating set. It suffices to show that $T$ is a subset of $S$. \par 

Let $g$ be a generator in subgroup $G$. Every element in the characteristic must be of the form $a_0+a_1+ \cdots + a_{q-1}$, where $a_i$ are non-negative integers and $a_0 + a_1g + \cdots + a_{q-1}g^{q-1}= 0$ in $\mathbb{F}_p$. However, we can reduce this equation. Let $\min (a_0, a_1, \cdots, a_{q-1}) = a_i$. Then,
\begin{align*}
    0 &= g^{q-i-1} \left(a_0 + a_1g + \cdots + a_{q-1}g^{q-1} \right)  \\
     &= a_0g^{q-i-1} + \cdots + a_ig^{q-1} + \cdots + a_{q-1}g^{2q-i-2} \\
    &= a_{i+1} + a_{i+2}g + \cdots + a_{q-1}g^{q-i-2} + a_0g^{q-i-1} + \cdots + a_ig^{q-1}\\
    &= (a_{i+1} - a_i) + (a_{i+2} - a_i)g + \cdots + (a_{i-1}-a_i)g^{q-2}, \\
\end{align*}
with all $a_j-a_i$ being non-negative by minimality of $a_i$.\par 

Note that this reduction yields all zero coefficients if either all the coefficients $a_i$ were originally 1) equal, or 2) divisible by $p$. (In other words, all $a_i$ are congruent modulo $p$.) If all coefficients were equal, then $a_0 + a_1 + \cdots + a_{q-1} = q \cdot a_0$, so $q$ is a generator. If all coefficients are $0$ mod $p$, then the sum of coefficients is some multiple of $p$, so $p$ is also included in the generating set. Note, however, that because usually $p$ is significantly large, $p$ is not in the minimal generating set for any $(p,n)$-hyperfield. We thus assume that $p$ is never in the minimal generating set, which we have verified computationally.\par

We claim that all other generators are produced from the sum of coefficients of some expression in terms of $g$ with at most $q-1$ terms.
\begin{lemma}\label{lem:qminusone}
Let $m \neq p,q$ be in the minimal generating set. Then $$m = b_0 + b_1 + \cdots + b_{q-2},$$ where $b_i$ are elements in $\mathbb{F}_p$ such that $b_0 + b_1g + \cdots + b_{q-2}g^{q-2} = 0$ in $\mathbb{F}_p$.
\end{lemma}
\begin{proof}
Suppose that $m = c_0 + \cdots + c_{q-1}$ such that $\gamma = c_0 + c_1g + \cdots + c_{q-1}g^{q-1} \equiv 0 \pmod{p}$, with $c_i$ being integers. If one coefficient $c_j$ is zero, then we can multiply $\gamma$ by some power of $g$, shift indices, and reach the desired result. Furthermore, if all coefficients are equal, then $m=q$, contradiction. Thus, assume all $c_i$ are positive and not all equal. \par 

We first prove that $c_i \in \mathbb{F}_p$. Let $c_i \equiv d_i \pmod{p}$.
\begin{align*}
    c_0 + c_1g + \cdots + c_{q-1}g^{q-1} &\equiv 0 \pmod{p} \\
    \implies d_0 + d_1g + \cdots + d_{q-1}g^{q-1} &\equiv 0 \pmod{p}.
\end{align*}
Let $m' = d_0 + \cdots + d_{q-1}$. Because $c_i \equiv d_i \pmod{p}$, $m = m' + kp$ for some non-negative integer $k$. Thus, $m$ can be generated by $m'$ and $p$, so $m$ cannot be in the minimal generating set. Thus, all coefficients $c_i$ must be in $\mathbb{F}_p$. \par 

Suppose again that $m = c_0 + \cdots + c_{q-1}$ such that $\gamma = c_0 + c_1g + \cdots + c_{q-1}g^{q-1} \equiv 0\pmod{p}$, with $0 \leq c_i < p$. Assume that all $c_i$ are positive and not all equal. Without loss of generality let $\min \{c_0, \cdots, c_{q-1}\} = c_{q-1}$ (if not, then multiply $\gamma$ by some power of $g$ and shift indices such that the coefficient of $g^{q-1}$ is minimal). We can use our reduction method shown earlier to get 
$$\gamma = (c_0 - c_{q-1}) + (c_1 - c_{q-1})g + \cdots + (c_{q-2} - c_{q-1})g^{q-2} = 0.$$
Note that all coefficients of $g^i$ are non-negative by minimality of $c_{q-1}$. Let $m'' = (c_0-c_{q-1}) + \cdots + (c_{q-2} + c_{q-1})$. We see that
$$m'' + q \cdot c_{q-1} = \left(\sum_{i=0}^{q-2} c_i-c_{q-1}\right) + q \cdot c_{q-1} = \sum_{i=0}^{q-1} c_i = m,$$
so $m$ can be generated by $m''$ and $q$. Therefore, $m$ cannot be in the minimal generating set, and all elements in the minimal generating set are either $q$ or expressed as the sum of at most $q-1$ terms.
\end{proof}\par 
For notation, we let $M = T \backslash \{q\}$, or the set of all generators in the minimal generating set that are not equal to $q$. From our theorem, it then suffices to prove that 
$$M \subseteq \left \{\left. \sum_{i=0}^{q-2} a_i \text{ } \right \vert N\left( {\sum_{i=0}^{q-2} a_i\zeta_q^i}\right)\equiv 0\pmod p \right \}$$
for $a_i \in \mathbb{F}_p$.\par 

We now define a norm analogue for $\mathbb{F}_p$.
\begin{defn}\label{def:pnorm}
Let $n,p$ be positive integers such that $p$ is prime and $n|p-1$. Let $f(g) = f_0 + f_1g + \cdots + f_kg^k \in \mathbb{F}_p$, where $g$ is a primitive $n$-th root in $\mathbb{F}_p$. The $p$-norm of $f(g)$ is the integer
$$N'(f(g)) = \prod_{i \in (\ZZ/n\ZZ)^{\times}} f(g^i).$$
\end{defn}
Let 
\begin{align*}
    \alpha_1(\zeta_n) &= a_0 + a_1\zeta_n + \cdots + a_k\zeta_n^k \in \mathbb{F}_p[\zeta_n] \\
    \alpha_2(g) &= a_0 + a_1g + \cdots + a_kg^k \in \mathbb{F}_p,
\end{align*} 
where $\zeta_n$ is a primitive $n$-th root of unity and $g$ is a primitive $n$-th root in $\mathbb{F}_p$. To compare $\Norm{\alpha_1}$ and $N'(\alpha_2)$, we prove the following lemma:
\begin{lemma}\label{lem:homomorph}
There exists a homomorphism $\rho : \mathbb{F}_p[\zeta_n] \rightarrow \mathbb{F}_p$ that maps $\zeta_n$ to $g$ and every element in $\mathbb{F}_p$ to itself.
\end{lemma} 
\begin{proof}
To show that a homomorphism exists, it suffices to prove the following statement.
\begin{lemma}\label{lem:linearmap}
Let $\zeta_n$ be a primitive $n$-th root of unity and $g$ be a primitive $n$-th root in $\mathbb{F}_p$. Let $V_1, V_2$ be vector spaces over field $\mathbb{F}_p$ with bases $\{1,\zeta_n, \cdots, \zeta_n^{n-1}\}$ and $\{1,g,\cdots, g^{n-1}\}$, respectively. Then there exists a linear map $f: V_1 \rightarrow V_2$ that maps $\zeta_n^k$ to $g^k$.
\end{lemma}
\begin{proof}
Let $\vec{v}_1$, $\vec{v}_2$ be vectors in $V_1$, and $c$ is in field $\mathbb{F}_p$. We will prove the following:
\begin{align}
    f(\vec{v}_1 + \vec{v}_2) &= f(\vec{v}_1) + f(\vec{v}_2) \label{eq:add}  \\
    f(c \vec{v}_1) &= c f(\vec{v}_2). \label{eq:scalar}
\end{align}
Let $\vec{v}_1 = a_0 + a_1\zeta_n + \cdots + a_{n-1}\zeta_n^{n-1}$ and $\vec{v}_2 = b_0 + b_1\zeta_n + \cdots + b_{n-1}\zeta_n^{n-1}$. The proof for $(1)$ is straightforward:
\begin{align*}
    f(\vec{v}_1 + \vec{v}_2) &= f((a_0+b_0) + (a_1+b_1)\zeta_n + \cdots + (a_{n-1} + b_{n-1})\zeta_n^{n-1}) \\
    &= (a_0+b_0) + (a_1+b_1)g + \cdots + (a_{n-1} + b_{n-1})g^{n-1} \\
    f(\vec{v}_1) + f(\vec{v}_2) &= (a_0 + a_1g + \cdots + a_{n-1}g^{n-1}) + (b_0 + b_1g + \cdots + b_{n-1}g^{n-1}) \\
    &= (a_0+b_0) + (a_1+b_1)g + \cdots + (a_{n-1}+b_{n-1})g^{n-1}
\end{align*}
Therefore $f(\vec{v}_1 + \vec{v}_2) = f(\vec{v}_1) + f(\vec{v}_2)$. \par

The proof for $(2)$ is also straightforward:
\begin{align*}
    f(c\vec{v}_1) &= f(ca_0 + ca_1\zeta_n + \cdots + ca_{n-1}\zeta_n^{n-1} \\
    &= c(a_0 + a_1g + \cdots + a_{n-1}g^{n-1}) \\
    cf(\vec{v}_1) &= cf(a_0+a_1\zeta_n + \cdots + a_{n-1}\zeta_n^{n-1}) \\
    &= c(a_0+a_1g + \cdots + a_{n-1}g^{n-1}),
\end{align*}
so $f(c\vec{v}_1) = cf(\vec{v}_1)$. \par 

Thus, $f$ is a linear map because we know how every vector in the basis of $V_1$ maps to $V_2$ and the two conditions for a linear map hold.
\end{proof}
Because of the existence of a linear map by Lemma \ref{lem:linearmap}, there exists a homomorphism from $\mathbb{F}_p[\zeta_n]$ to $\mathbb{F}_p$ that maps $\zeta_n$ to $g$. We can check that $\rho$ preserves multiplication. Let
\begin{align*}
    x &= x_0 + x_1\zeta_n + \cdots + x_{n-1}\zeta_n^{n-1} \\
    y &= y_0 + y_1\zeta_n + \cdots + y_{n-1}\zeta_n^{n-1}, \\
\end{align*}
with $x,y \in \mathbb{F}_p[\zeta_n]$. Then
\begin{align*}
    \rho (x \cdot y) &= \rho \left(\left(x_0 + \cdots + x_{n-1}\zeta_n^{n-1} \right)\left(y_0 + \cdots + y_{n-1}\zeta_n^{n-1} \right) \right) \\ &=
    \rho \left( \sum_{i=0}^{2n-2} \sum_{\substack{j+k=i \\0 \leq j,k < n}} x_jy_k\zeta_n^i \right) \\ &= \sum_{i=0}^{2n-2} \sum_{\substack{j+k=i \\0 \leq j,k < n}} x_jy_kg^i \\ &= (x_0 + \cdots + x_{n-1}g^{n-1})(y_0 + \cdots + y_{n-1}g^{n-1})\\ &= \rho(x) \cdot \rho(y),
\end{align*}
so $\rho(x\cdot y) = \rho (x) \cdot \rho (y)$ and $\rho$ preserves multiplication.\end{proof}
Let $i \in (\ZZ/q\ZZ)^{\times}$. By Lemma \ref{lem:homomorph}, if $\rho$ is the homomorphism, then $\rho(\alpha_1(\zeta_q^i)) = \alpha_2(g^i)$. Taking the product over all units $i$ in $\ZZ/q\ZZ$, $\rho\left(\Norm{\alpha_1(\zeta_q)}\right) = N'(\alpha_2(g))$. But $\Norm{\alpha_1(\zeta_q)}$ is an integer, so by definition of $\rho$ in Lemma $\ref{lem:homomorph}$,
\begin{equation}
    \Norm{\alpha_1(\zeta_q)} \equiv N'(\alpha_2(g)) \pmod{p}. \label{eq:congruent}
\end{equation}
Suppose $m = b_0 + b_1 + \cdots + b_{q-2}$ is in the generating set for $b_i \in \mathbb{F}_p$. Let $\beta = b_0 + b_1g + \cdots + b_{q-2}g^{q-2} \equiv 0 \pmod{p} \implies p | \beta$. By Definition \ref{def:pnorm},
$$\beta | N'(\beta).$$
Let $\beta (\zeta_q) = b_0 + b_1 \zeta_q + \cdots + b_{q-2} \zeta_q^{q-2}$. By Equation \ref{eq:congruent}, 
\begin{align*}
    N'(\beta) \equiv \Norm{\beta (\zeta_q)} \pmod{p} \\
    \implies p | \beta | N'(\beta) \equiv \Norm{\beta (\zeta_q)} \pmod{p} \\
    \implies \Norm{\beta(\zeta_q)} \equiv 0 \pmod{p}.
\end{align*}
This is true for all generators $m \in M$, so
$$M \subseteq \left \{\left. \sum_{i=0}^{q-2} a_i \text{ } \right \vert N\left( {\sum_{i=0}^{q-2} a_i\zeta_q^i}\right)\equiv 0\pmod p \right \}$$
for $a_i \in \mathbb{F}_p$.\end{proof}




Theorem \ref{thm:genprime} is significant in the following two ways:
\begin{enumerate}
    \item For every hyperfield of type $(p,q)$, with $q$ prime, there exists a homogeneous polynomial $f$ of degree $q-1$ whose solutions produce a generating set, and
    \item The criterion is finite: there are at most $p^{q-1}$ tuples $(a_0, \dots, a_{q-2})$ to test.
\end{enumerate}
In particular, the homogeneous polynomial $f$ whose solutions produce a generating set is the polynomial in $q-1$ variables $a_0,a_1, \cdots, a_{q-2}$ that computes the norm of $a_0 + a_1 \zeta_q + \cdots + a_{q-2} \zeta_q^{q-2}$.\\

We can improve on the number of tuples to test in $(2)$. First, we assume that every generator in the minimal generating set is less than $p$, so it suffices to check only the tuples whose elements sum to some positive value less than $p$. Furthermore, $q$ is always a generator, so we may exclude all tuples whose elements sum to any multiples of $q$. By simple counting arguments, the total number of tuples we must test is at most $$\dbinom{p+q-2}{q-1} - \sum_{k=0}^{\frac{p-q-1}{q}} \dbinom{q-2+kq}{q-2}.$$ \par 


\appendix 
\section{Computational Data} \label{sec:appen1}
\tiny
\centering
\begin{tabular} {|c| c| c| c| c| c| c| c| c| c| c| c| c| c| c|}\hline
$p$	&	$|G|$	&	Gen. Set	&	$p$	&	$|G|$	&	Gen. Set	&	$p$	&	$|G|$	&	Gen. Set	&	$p$	&	$|G|$	&	Gen. Set	&	$p$	&	$|G|$	&	Gen. Set	\\\hline
2	&	1	&	{ 2 }	&	43	&	42	&	{ 2 }	&	89	&	2	&	{ 2 }	&	127	&	42	&	{ 2 }	&	173	&	4	&	{ 2 }	\\
3	&	1	&	{ 3 }	&	43	&	21	&	{ 3 5 }	&	89	&	11	&	{ 4 6 7 }	&	127	&	2	&	{ 2 }	&	173	&	2	&	{ 2 }	\\
3	&	2	&	{ 2 }	&	43	&	14	&	{ 2 }	&	89	&	22	&	{ 2 }	&	131	&	1	&	{ 131 }	&	173	&	86	&	{ 2 }	\\
5	&	1	&	{ 5 }	&	47	&	1	&	{ 47 }	&	89	&	44	&	{ 2 }	&	131	&	65	&	{ 3 5 }	&	173	&	172	&	{ 2 }	\\
5	&	2	&	{ 2 }	&	47	&	23	&	{ 3 5 }	&	89	&	8	&	{ 2 }	&	131	&	2	&	{ 2 }	&	173	&	43	&	{ 3 5 }	\\
5	&	4	&	{ 2 }	&	47	&	2	&	{ 2 }	&	89	&	4	&	{ 2 }	&	131	&	10	&	{ 2 }	&	179	&	1	&	{ 179 }	\\
7	&	1	&	{ 7 }	&	47	&	46	&	{ 2 }	&	97	&	4	&	{ 2 }	&	131	&	26	&	{ 2 }	&	179	&	178	&	{ 2 }	\\
7	&	3	&	{ 3 5 }	&	53	&	1	&	{ 53 }	&	97	&	12	&	{ 2 }	&	131	&	13	&	{ 4 6 7 }	&	179	&	2	&	{ 2 }	\\
7	&	6	&	{ 2 }	&	53	&	52	&	{ 2 }	&	97	&	48	&	{ 2 }	&	131	&	5	&	{ 5 9 11 12 }	&	179	&	89	&	{ 3 5 }	\\
7	&	2	&	{ 2 }	&	53	&	13	&	{ 3 5 }	&	97	&	16	&	{ 2 }	&	131	&	130	&	{ 2 }	&	181	&	60	&	{ 2 }	\\
11	&	1	&	{ 11 }	&	53	&	2	&	{ 2 }	&	97	&	6	&	{ 2 }	&	137	&	136	&	{ 2 }	&	181	&	180	&	{ 2 }	\\
11	&	5	&	{ 3 5 }	&	53	&	4	&	{ 2 }	&	97	&	96	&	{ 2 }	&	137	&	4	&	{ 2 }	&	181	&	36	&	{ 2 }	\\
11	&	10	&	{ 2 }	&	53	&	26	&	{ 2 }	&	97	&	2	&	{ 2 }	&	137	&	1	&	{ 137 }	&	181	&	12	&	{ 2 }	\\
11	&	2	&	{ 2 }	&	59	&	1	&	{ 59 }	&	97	&	1	&	{ 97 }	&	137	&	34	&	{ 2 }	&	181	&	18	&	{ 2 }	\\
13	&	1	&	{ 13 }	&	59	&	2	&	{ 2 }	&	97	&	3	&	{ 3 19 }	&	137	&	17	&	{ 4 6 7 }	&	181	&	30	&	{ 2 }	\\
13	&	12	&	{ 2 }	&	59	&	58	&	{ 2 }	&	97	&	24	&	{ 2 }	&	137	&	2	&	{ 2 }	&	181	&	10	&	{ 2 }	\\
13	&	2	&	{ 2 }	&	59	&	29	&	{ 3 5 }	&	97	&	8	&	{ 2 }	&	137	&	8	&	{ 2 }	&	181	&	20	&	{ 2 }	\\
13	&	3	&	{ 3 7 }	&	61	&	30	&	{ 2 }	&	97	&	32	&	{ 2 }	&	137	&	68	&	{ 2 }	&	181	&	1	&	{ 181 }	\\
13	&	4	&	{ 2 }	&	61	&	15	&	{ 3 5 }	&	101	&	1	&	{ 101 }	&	139	&	1	&	{ 139 }	&	181	&	9	&	{ 3 8 }	\\
13	&	6	&	{ 2 }	&	61	&	3	&	{ 3 14 }	&	101	&	50	&	{ 2 }	&	139	&	23	&	{ 4 6 7 }	&	181	&	6	&	{ 2 }	\\
17	&	1	&	{ 17 }	&	61	&	5	&	{ 4 7 }	&	101	&	25	&	{ 3 5 }	&	139	&	46	&	{ 2 }	&	181	&	3	&	{ 3 26 }	\\
17	&	2	&	{ 2 }	&	61	&	1	&	{ 61 }	&	101	&	2	&	{ 2 }	&	139	&	3	&	{ 3 23 }	&	181	&	45	&	{ 3 5 }	\\
17	&	16	&	{ 2 }	&	61	&	10	&	{ 2 }	&	101	&	20	&	{ 2 }	&	139	&	69	&	{ 3 5 }	&	181	&	15	&	{ 3 7 }	\\
17	&	4	&	{ 2 }	&	61	&	60	&	{ 2 }	&	101	&	4	&	{ 2 }	&	139	&	138	&	{ 2 }	&	181	&	5	&	{ 5 7 }	\\
17	&	8	&	{ 2 }	&	61	&	6	&	{ 2 }	&	101	&	100	&	{ 2 }	&	139	&	2	&	{ 2 }	&	181	&	4	&	{ 2 }	\\
19	&	1	&	{ 19 }	&	61	&	20	&	{ 2 }	&	101	&	10	&	{ 2 }	&	139	&	6	&	{ 2 }	&	181	&	90	&	{ 2 }	\\
19	&	18	&	{ 2 }	&	61	&	12	&	{ 2 }	&	101	&	5	&	{ 5 8 9 11 }	&	149	&	1	&	{ 149 }	&	181	&	2	&	{ 2 }	\\
19	&	3	&	{ 3 8 }	&	61	&	2	&	{ 2 }	&	103	&	1	&	{ 103 }	&	149	&	2	&	{ 2 }	&	191	&	38	&	{ 2 }	\\
19	&	2	&	{ 2 }	&	61	&	4	&	{ 2 }	&	103	&	102	&	{ 2 }	&	149	&	4	&	{ 2 }	&	191	&	5	&	{ 5 8 9 11 }	\\
19	&	6	&	{ 2 }	&	67	&	1	&	{ 67 }	&	103	&	3	&	{ 3 20 }	&	149	&	148	&	{ 2 }	&	191	&	19	&	{ 3 5 }	\\
19	&	9	&	{ 3 5 }	&	67	&	33	&	{ 3 5 }	&	103	&	6	&	{ 2 }	&	149	&	37	&	{ 3 5 }	&	191	&	1	&	{ 191 }	\\
23	&	1	&	{ 23 }	&	67	&	6	&	{ 2 }	&	103	&	17	&	{ 3 5 }	&	149	&	74	&	{ 2 }	&	191	&	10	&	{ 2 }	\\
23	&	2	&	{ 2 }	&	67	&	3	&	{ 3 16 }	&	103	&	34	&	{ 2 }	&	151	&	150	&	{ 2 }	&	191	&	190	&	{ 2 }	\\
23	&	22	&	{ 2 }	&	67	&	11	&	{ 4 6 7 }	&	103	&	2	&	{ 2 }	&	151	&	10	&	{ 2 }	&	191	&	95	&	{ 3 5 }	\\
23	&	11	&	{ 3 5 }	&	67	&	2	&	{ 2 }	&	103	&	51	&	{ 3 5 }	&	151	&	3	&	{ 3 23 }	&	191	&	2	&	{ 2 }	\\
29	&	1	&	{ 29 }	&	67	&	66	&	{ 2 }	&	107	&	1	&	{ 107 }	&	151	&	5	&	{ 5 9 11 13 }	&	193	&	6	&	{ 2 }	\\
29	&	7	&	{ 4 6 7 }	&	67	&	22	&	{ 2 }	&	107	&	106	&	{ 2 }	&	151	&	75	&	{ 3 5 }	&	193	&	3	&	{ 3 25 }	\\
29	&	2	&	{ 2 }	&	71	&	1	&	{ 71 }	&	107	&	2	&	{ 2 }	&	151	&	25	&	{ 3 5 }	&	193	&	2	&	{ 2 }	\\
29	&	4	&	{ 2 }	&	71	&	10	&	{ 2 }	&	107	&	53	&	{ 3 5 }	&	151	&	1	&	{ 151 }	&	193	&	192	&	{ 2 }	\\
29	&	28	&	{ 2 }	&	71	&	70	&	{ 2 }	&	109	&	27	&	{ 3 5 }	&	151	&	50	&	{ 2 }	&	193	&	8	&	{ 2 }	\\
29	&	14	&	{ 2 }	&	71	&	2	&	{ 2 }	&	109	&	18	&	{ 2 }	&	151	&	30	&	{ 2 }	&	193	&	96	&	{ 2 }	\\
31	&	1	&	{ 31 }	&	71	&	35	&	{ 3 5 }	&	109	&	54	&	{ 2 }	&	151	&	15	&	{ 3 5 }	&	193	&	64	&	{ 2 }	\\
31	&	5	&	{ 5 7 8 9 }	&	71	&	7	&	{ 4 6 7 }	&	109	&	4	&	{ 2 }	&	151	&	6	&	{ 2 }	&	193	&	1	&	{ 193 }	\\
31	&	10	&	{ 2 }	&	71	&	14	&	{ 2 }	&	109	&	1	&	{ 109 }	&	151	&	2	&	{ 2 }	&	193	&	32	&	{ 2 }	\\
31	&	3	&	{ 3 11 }	&	71	&	5	&	{ 5 7 8 9 }	&	109	&	12	&	{ 2 }	&	157	&	6	&	{ 2 }	&	193	&	48	&	{ 2 }	\\
31	&	6	&	{ 2 }	&	73	&	3	&	{ 3 17 }	&	109	&	2	&	{ 2 }	&	157	&	39	&	{ 3 5 }	&	193	&	16	&	{ 2 }	\\
31	&	2	&	{ 2 }	&	73	&	4	&	{ 2 }	&	109	&	6	&	{ 2 }	&	157	&	2	&	{ 2 }	&	193	&	12	&	{ 2 }	\\
31	&	30	&	{ 2 }	&	73	&	18	&	{ 2 }	&	109	&	36	&	{ 2 }	&	157	&	52	&	{ 2 }	&	193	&	4	&	{ 2 }	\\
31	&	15	&	{ 3 5 }	&	73	&	6	&	{ 2 }	&	109	&	9	&	{ 3 7 }	&	157	&	78	&	{ 2 }	&	193	&	24	&	{ 2 }	\\
37	&	1	&	{ 37 }	&	73	&	36	&	{ 2 }	&	109	&	108	&	{ 2 }	&	157	&	1	&	{ 157 }	&	197	&	4	&	{ 2 }	\\
37	&	36	&	{ 2 }	&	73	&	24	&	{ 2 }	&	109	&	3	&	{ 3 19 }	&	157	&	13	&	{ 4 6 7 }	&	197	&	7	&	{ 5 7 8 9 }	\\
37	&	12	&	{ 2 }	&	73	&	1	&	{ 73 }	&	113	&	8	&	{ 2 }	&	157	&	26	&	{ 2 }	&	197	&	49	&	{ 3 5 }	\\
37	&	2	&	{ 2 }	&	73	&	72	&	{ 2 }	&	113	&	28	&	{ 2 }	&	157	&	3	&	{ 3 25 }	&	197	&	28	&	{ 2 }	\\
37	&	6	&	{ 2 }	&	73	&	8	&	{ 2 }	&	113	&	4	&	{ 2 }	&	157	&	12	&	{ 2 }	&	197	&	1	&	{ 197 }	\\
37	&	9	&	{ 3 5 }	&	73	&	2	&	{ 2 }	&	113	&	56	&	{ 2 }	&	157	&	156	&	{ 2 }	&	197	&	2	&	{ 2 }	\\
37	&	3	&	{ 3 11 }	&	73	&	12	&	{ 2 }	&	113	&	1	&	{ 113 }	&	157	&	4	&	{ 2 }	&	197	&	196	&	{ 2 }	\\
37	&	18	&	{ 2 }	&	73	&	9	&	{ 3 5 }	&	113	&	2	&	{ 2 }	&	163	&	9	&	{ 3 7 }	&	197	&	14	&	{ 2 }	\\
37	&	4	&	{ 2 }	&	79	&	1	&	{ 79 }	&	113	&	7	&	{ 5 7 8 9 }	&	163	&	2	&	{ 2 }	&	197	&	98	&	{ 2 }	\\
41	&	5	&	{ 5 7 8 9 }	&	79	&	2	&	{ 2 }	&	113	&	16	&	{ 2 }	&	163	&	162	&	{ 2 }	&	199	&	9	&	{ 3 7 }	\\
41	&	10	&	{ 2 }	&	79	&	39	&	{ 3 5 }	&	113	&	112	&	{ 2 }	&	163	&	18	&	{ 2 }	&	199	&	2	&	{ 2 }	\\
41	&	40	&	{ 2 }	&	79	&	13	&	{ 4 6 7 }	&	113	&	14	&	{ 2 }	&	163	&	1	&	{ 163 }	&	199	&	99	&	{ 3 5 }	\\
41	&	1	&	{ 41 }	&	79	&	6	&	{ 2 }	&	127	&	6	&	{ 2 }	&	163	&	27	&	{ 3 5 }	&	199	&	11	&	{ 6 8 9 10 11 }	\\
41	&	8	&	{ 2 }	&	79	&	26	&	{ 2 }	&	127	&	21	&	{ 3 5 }	&	163	&	81	&	{ 3 5 }	&	199	&	6	&	{ 2 }	\\
41	&	2	&	{ 2 }	&	79	&	78	&	{ 2 }	&	127	&	1	&	{ 127 }	&	163	&	6	&	{ 2 }	&	199	&	1	&	{ 199 }	\\
41	&	20	&	{ 2 }	&	79	&	3	&	{ 3 17 }	&	127	&	126	&	{ 2 }	&	163	&	54	&	{ 2 }	&	199	&	33	&	{ 3 5 }	\\
41	&	4	&	{ 2 }	&	83	&	1	&	{ 83 }	&	127	&	9	&	{ 3 }	&	163	&	3	&	{ 3 25 }	&	199	&	198	&	{ 2 }	\\\hline
\end{tabular}
\newpage

\tiny
\centering
\begin{tabular} {|c| c| c| c| c| c| c| c| c| c| c| c| c| c| c|}\hline
$p$	&	$|G|$	&	Gen. Set	&	$p$	&	$|G|$	&	Gen. Set	&	$p$	&	$|G|$	&	Gen. Set	&	$p$	&	$|G|$	&	Gen. Set	&	$p$	&	$|G|$	&	Gen. Set	\\\hline
43	&	1	&	{ 43 }	&	83	&	82	&	{ 2 }	&	127	&	63	&	{ 3 5 }	&	167	&	1	&	{ 167 }	&	199	&	66	&	{ 2 }	\\
43	&	3	&	{ 3 13 }	&	83	&	2	&	{ 2 }	&	127	&	3	&	{ 3 20 }	&	167	&	83	&	{ 3 5 }	&	199	&	3	&	{ 3 28 }	\\
43	&	6	&	{ 2 }	&	83	&	41	&	{ 3 5 }	&	127	&	7	&	{ 7 9 10 11 12 13 }	&	167	&	2	&	{ 2 }	&	199	&	18	&	{ 2 }	\\
43	&	7	&	{ 3 7 }	&	89	&	1	&	{ 89 }	&	127	&	14	&	{ 2 }	&	167	&	166	&	{ 2 }	&	199	&	22	&	{ 2 }	\\
43	&	2	&	{ 2 }	&	89	&	88	&	{ 2 }	&	127	&	18	&	{ 2 }	&	173	&	1	&	{ 173 }	&		&		&		\\\hline
\end{tabular}
\end{document}